\documentclass[a4paper,12pt]{article}

\usepackage[austrian,american]{babel}

\usepackage[intlimits,leqno]{amsmath}
\usepackage{amssymb} 
\usepackage{amsthm}
\usepackage{graphicx,color}
\usepackage{makeidx}
\usepackage{epsfig}
\usepackage{latexsym}
\usepackage{pstricks,pst-node,amsfonts}

\theoremstyle{plain}
\newtheorem{theo}{Theorem}
\newtheorem{prop}{Proposition}

\newtheorem{lemm}{Lemma}
\newtheorem{rema}{Remark}
\newtheorem{defi}{Definition}

\newtheorem{exam}{Example}

\newcommand\N{{\mathbb{N}}}
\newcommand\R{{\mathbb{R}}}
\newcommand\C{{\mathbb{C}}}
\newcommand{\K}{\mathcal{K}}

\newcommand{\F}{\mathcal{F}}
\newcommand{\T}{\mathcal{T}}

\renewcommand{\i}{{\rm i}}

\newcommand{\x}{\mathbf{x}}

\renewcommand{\d}{{\rm d}}
\newcommand{\supp}{ { \rm{supp} } }

\renewcommand\T{ T }
\renewcommand\t{ \tau }

\begin{document}

\title{A representation formula for solutions of second order ode's with time dependent coefficients 
and its application to model dissipative oscillations and waves}


\author{ Richard Kowar\\
Department of Mathematics, University of Innsbruck, \\
Technikerstrasse 21a, A-6020, Innsbruck, Austria
}

\maketitle

\begin{abstract}
In this paper, we model, classify and investigate the solutions of (normalized) second order ode's with 
\emph{nonconstant continuous coefficients}. We introduce a generalized \emph{frequency function} as the 
solution of a \emph{nonlinear integro-differential equation}, show its existence and then derive a 
representation formula for (all) solutions of (normalized) second order ode's with \emph{nonconstant 
continuous coefficients}. Because this formula specifies the interplay between the coefficients of the 
ode, the \emph{relaxation function} ("strongly" decreasing positive function) and the frequency function 
of the oscillation, it can be applied to design models of dissipative oscillations. As an application, 
we present and discuss some oscillation models that stop within a finite time period. 
Moreover, we demonstrate that a large class of oscillations can be used to design and analyze dissipative 
waves. In particular, it is easy to model dissipative waves that cause in each point of space an oscillation 
that stops after a finite time period. 
\end{abstract}

\section{Introduction}

The first main result of this paper is the derivation of a \emph{representation formula} for solutions of (normalized) second order ode's with \emph{time dependent coefficients} $a$ and $b$, 
i.e. solutions of 
\begin{equation}\label{2orderode}
\begin{aligned}
    v'' + a\,v' + b\,v = f_0 \quad\mbox{on}\quad (0,\T) \quad\mbox{with} \quad
    v(0) = \varphi \quad\mbox{and}\quad v'(0) = \psi \,,
\end{aligned}
\end{equation}
where $\T\in (0,\infty]$,  $a,\,b:[0,\T)\to\R$ are continuous functions, $f_0$ is continuous satisfying 
$\supp(f_0)\subset (0,\T)$ and $\varphi,\,\psi\in\R$. 
For this purpose, we show that there exists a differentiable \emph{frequency function} $\omega:[0,\T) \to\R$ 
satisfying a \emph{nonlinear integro-differential equation} such that the 
solution of~(\ref{2orderode}) (with $f_0=0$) reads as follows 
\begin{equation}\label{GenSolOscill}
   v(t) = \left[ \varphi\,\cos\left( \tilde\omega(t)\right) 
                 + \left( \psi + \alpha(0)\,\varphi\right)\,
                 \frac{\sin\left( \tilde\omega(t)\right)}{\omega(0)}          
          \right]\,\varrho(t) \,,
\end{equation}
where 
\begin{equation}\label{deftildeop}
\begin{aligned}
   &\alpha:= \frac{a}{2} + \frac{\omega'}{2\,\omega}  \,,
              \quad
    \varrho := \exp\left( -\tilde\alpha \right)\,H
               \quad\mbox{and}  \\
  &\tilde \,\quad \mbox{ is the operator defined by} \quad \tilde g(t) := \int_0^t g(s) \,\d s 
\end{aligned}
\end{equation}
for $t\in [0,\T)$ and integrable $g$; $\omega'$ and $H$ denote time derivative of $\omega$ and the 
\emph{Heaviside} function, respectively. We call $\varrho$ and $\alpha$ the \emph{relaxation function} and 
the \emph{dissipation} function, respectively.  
It turns out, if $a$ and $b$ are continuous on $[0,\T)$, then the existence of this frequency function 
is always guaranteed. If $\omega'$ does not vanish, then the \emph{dissipation} $\alpha$ (of the oscillation) 
also depends on the \emph{frequency} $\omega$ and not only on the \emph{damping coefficient} $\alpha$.  
We note that there also exists a formula for $f_0 \not= 0$ (cf. Theorem~\ref{theo:oscillmainb} below). 
This representation formula is a useful tool for designing and analyzing all types 
of dissipative oscillations. Indeed, from~(\ref{GenSolOscill}), ~(\ref{deftildeop}) and the equation 
for $\omega$ (cf. (\ref{defomega}) below), we see the interplay between the coefficients $a$ and $b$ in 
the ode and the relaxation function $\varrho$ and the frequency function $\omega$ of the oscillation. 
We conclude our analysis of~(\ref{2orderode}) with examples in which we design and discuss 
oscillations that stop within a finite time period.

In the second part of this paper, we demonstrate that special semigroups of oscillations can be used 
to design and analyze dissipative waves. 
For example, a spherical dissipative pressure wave can be modeled by (cf.~\cite{KoSc12}) 
\begin{equation}\label{modelG}
      G(x,t) 
          := \frac{\K_{|x|}\left( t-\frac{|x|}{c_0} \right)}{4\,\pi\,|x|} \qquad\quad \mbox{($c_0$ constant)}\, 
\end{equation}
where 
\begin{equation}\label{modelKR}
     \F(\K_R)(\omega) := e^{-\alpha(\omega)\,R} \qquad\mbox{ for } \quad \mbox{$\omega\in\R$ and $R\geq 0$.}
\end{equation}
For basic facts about dissipative waves we refer for example 
to~\cite{NacSmiWaa90,Sz94,Sz95,KiFrCoSa00,WaHuBrMi00,We00,HanSer03,CheHolm04,WaMoMi05,PatGre06,KoSc12}. 
It seems evident that the elongation, caused by the spherical wave $G$ in a space point $x_0$ 
with $|x_0|=1$, corresponds to an "oscillation" $t\mapsto \frac{\K_1\left( t-\frac{1}{c_0} \right)}{4\,\pi}$ 
that starts at time $\frac{1}{c_0}$ and ends at time $\T_1 :=\frac{1}{c_0} + \T$ for some 
$\T\in (0,\infty]$. 
Actually, due to $G(0,t) = \F^{-1}(1) = \delta(t)$, $\K_R$ may have a pole at $t=0$ for sufficiently small 
$R>0$, i.e. $\K_R$ may not be a proper oscillation for small $R>0$ (cf. Example~\ref{exam:Instr}).\footnote{It 
seems not reasonable that $t\mapsto G(x_0,t)$ has a pole for each $x_0\in\R^3$, but we do not investigate that 
issue in this paper. } 
At any rate, it stands to reason that there exists a class of oscillations, say $\{\K_*\in \mathcal{C}_{osc}\}$, 
that determines a class of dissipation laws $\{\alpha_* \in \mathcal{C}_{diss}\}$ for waves, 
due to
\begin{equation}\label{ReqForm}
     \F(\K_*) = e^{-\alpha_*\,|x_0|}\,  \qquad\mbox{for some}\qquad |x_0| = R_0 >0\,.
\end{equation}
Here $R_0$ may be different for different $\K_*$. If $\K_*$ or $\alpha$ is given, then the associated 
\emph{semigroup of oscillations} is defined by 
$$
    \left( \F^{-1} \left(\F(\K_*)^\frac{R}{R_0}\right) \right)_{R>0} 
       \quad\mbox{or equivalently by} \quad 
    \left( \F^{-1} \left(e^{-\alpha\,\frac{R}{R_0}}\right) \right)_{R>0}\,.
$$
We investigate the properties that an oscillation $\K_*$ must satisfy such that~(\ref{ReqForm}) is satisfied 
for $R_0=1$ and present some examples. \\

This paper is organized as follows: In Section~\ref{sec-GenOscill} we introduce the concept of frequency 
function related to~(\ref{2orderode}) and show its existence under fairly weak assumptions. Then we tackle 
the representation formula and the classification of general dissipative oscillations. This section is concluded 
with some examples. In Section~\ref{sec-DissWaves}, we investigate sufficient conditions for the density and 
the frequency function such that the respective oscillation determines a dissipation law for dissipative waves.  
In the appendix, we summarize some notations about the Fourier transform and present a small 
application of the Paley-Wiener-Schwartz Theorem that is used in our analysis. 
Finally, the paper is closed with the section Conclusion.


\section{Derivation of the representation formular and its application}
\label{sec-GenOscill}

In this section, we introduce the concept of a \emph{frequency function} $\omega:(0,\T)\to \R \cup
 \i\R$ 
with $\T\in (0,\infty]$ (and its extension to the real line), which generalizes the frequency and 
which is a solution of an \emph{integro differential equation}. 
Then we show its existence and derive the general representation formula for dissipative oscillations 
that make use of the frequency function. Moreover, we give a \emph{classification} of damped oscillations 
and apply the formula to model dissipative oscillations that stop at a finite time 
$\T>0$. 

According to Chapter VIII in~\cite{He91}, the second order ode with \emph{time dependent} 
coefficients~(\ref{2orderode}) has always a twice differentiable solution if $a$ and $b$ are continuous on 
$(0,\T)$. Here $\T\in (0,\infty]$ denotes the \emph{stopping time} of the oscillation, which may be finite 
or infinite.

\subsection{The frequency function and the solution formula}

\begin{defi}\label{defi:frequency}
Let $\T\in (0,\infty]$ and $a,\,b:[0,\T)\to [0,\infty)$ be continuous. If there exists a 
differentiable function $\omega:(0,\T)\to \R \cup \i\R$ satisfying 
\begin{equation}\label{defomega}
\begin{aligned}
    \int_0^t \omega^2(s) \,\d s 
          &= \int_0^t \left[ b(s) 
                              - \left( \frac{a(s)}{2} \right)^2 
                              + \left( \frac{\omega'}{2\,\omega}\right)^2(s)\right]\,\d s \\
          &\quad  - \frac{a(t)}{2} + \frac{a(0)}{2} 
                  - \frac{\omega'}{2\,\omega} (t) + \frac{\omega'}{2\,\omega}(0) \,,
\end{aligned}
\end{equation}
then we call $\omega$ the frequency function of the process $v$ modeled by~(\ref{2orderode}). 
Moreover, we define $\omega(0) := \omega(0+)$ and $\omega(\T) := \omega(\T-)$. \\
For applications involving the Fourier transform, it is convenient to extend $\omega$ to the real line by 
$\omega(t) = \omega(\T)$ for $t>\T$ and $\omega(t):=\omega(-t)$ for $t<0$.
\end{defi}

If the coefficient $a$ is differentiable, then the condition for the frequency function~(\ref{defomega}) simplifies to 
\begin{equation}\label{defomega1}
      \omega^2 
         - \left( \frac{\omega'}{2\,\omega} \right)^2
         + \left( \frac{\omega'}{2\,\omega} \right)'
         =  b
                - \left( \frac{a}{2} \right)^2 
                - \frac{a'}{2}  \,,
\end{equation}
and for the special case $\omega'=0$, we arrive at 
$$
    \omega(t) =  \sqrt{ b(t) - \left( \frac{a(t)}{2} \right)^2 - \frac{a'(t)}{2} } = \omega(0) 
        \qquad\mbox{for}\qquad  t\in(0,\T)\,.
$$

In the following Lemma and Theorem, we show that a \emph{frequency function} exists if the coefficients $a$ and 
$b$ are contiunous and prove the claimed representation formula for the solution of problem~(\ref{2orderode}).

\begin{lemm}\label{lemm:oscillmain}
Let $\omega_0 >0$, $\omega_1 \in \R$, $a,\,b:[0,\T)\to \R$ be continuous,   
\begin{itemize}
\item $v_1$ denote the solution of~(\ref{2orderode}) with $\varphi := 0$ and $\psi := 1$, 

\item $v_2$ denote the solution of~(\ref{2orderode}) with $\varphi := \frac{1}{\omega_0}$ 
      and~\footnote{Cf. Formula in Theorem~\ref{theo:oscillmain} below.} 
      $\psi := -\frac{1}{\omega_0}\, \left(\frac{a(0)}{2}+\frac{\omega_1}{2\,\omega_0}\right)$. 
\end{itemize} 
Then $v_1$ and $v_2$ exist and are twice continuously differentiable on $(0,\T)$. 
If $v_1^2(0) + v_2^2(0)>0$, then 
\begin{equation}\label{ansatzomega1}
    \omega(t) := \frac{v_1'(t)\,v_2(t) - v_1(t)\,v_2'(t)}{v_1^2(t) + v_2^2(t)} 
       \qquad\mbox{for}\qquad t\in I_\omega\,
\end{equation}
is a real valued frequency function on $I_\omega := \{t\in [0,\T)\,|\, v_1^2(t) + v_2^2(t) \not= 0\}$ and 
satisfies $\omega(0)=\omega_0>0$ and $\omega'(0)=\omega_1$. If $v_2^2(0) - v_1^2(0)>0$, then
\begin{equation}\label{ansatzomega2}
   \i\,\epsilon(t) :=  \omega(t) := \i\,\frac{v_1'(t)\,v_2(t) - v_1(t)\,v_2'(t)}{v_2^2(t) - v_1^2(t)} 
       \qquad\mbox{for}\qquad t\in I_\epsilon\,
\end{equation}
is an imaginary valued frequency function on $I_\epsilon := \{t\in [0,\T)\,|\, v_2^2(t) - v_1^2(t) \not= 0\}$ 
and $\epsilon(0)=\omega_0>0$ and $\epsilon'(0)=\omega_1$. 
\end{lemm}

\begin{proof}
Because of the continuity of $a$ and $b$, the solutions $v_1$ and $v_2$ exist and are twice continuously 
differentiable and real valued. In this proof, we denote the distributive derivative of a continuous function 
$f$ by $f'$. Then the proof is more instructive. \\
a) Let $v_1^2(0) + v_2^2(0)>0$. We have to show that 
$$
  X_\omega :=
    \omega^2 - \left( \frac{\omega'}{2\,\omega} \right)^2 + \left( \frac{\omega'}{2\,\omega} \right)' 
       - b + \left(\frac{a}{2}\right)^2 + \frac{a'}{2} 
   = 0 \qquad\mbox{on} \qquad I_\omega\,,
$$
From~(\ref{ansatzomega1}) together with $v_j'' = -a\,v_j' - b\,v_j$ for $j=1,\,2$, we obtain 
$$
   \frac{\omega'}{2\,\omega} = - \frac{v_1\,v_1'+v_2\,v_2'}{v_1^2+v_2^2} - \frac{a}{2}\,,
$$
which implies 
$$
   \left( \frac{\omega'}{2\,\omega} \right)^2 
          = \frac{a^2}{4} 
            + a\,\frac{v_1\,v_1'+v_2\,v_2'}{v_1^2+v_2^2} 
            + \frac{(v_1\,v_1'+v_2\,v_2')^2}{(v_1^2+v_2^2)^2}
$$
and 
$$
  \left( \frac{\omega'}{2\,\omega} \right)' 
          = b 
             - \omega^2 
             - \frac{a'}{2} 
             + a\,\frac{v_1\,v_1' + v_2\,v_2'}{v_1^2+v_2^2} 
             + \frac{(v_1\,v_1' + v_2\,v_2')^2}{v_1^2+v_2^2} \,.
$$
Consequently,
$$
\left( \frac{\omega'}{2\,\omega} \right)'-\left( \frac{\omega'}{2\,\omega} \right)^2 
          = b - \omega^2 - \left(\frac{a}{2}\right)^2 - \frac{a'}{2}
$$
and  therefore  
\begin{equation}\label{Xomega}
\begin{aligned}
   X_\omega = 0   \qquad\mbox{on} \qquad I_\omega\,.
\end{aligned}
\end{equation} 
From the definition of $v_1$ and $v_2$, it follows that 
$v_1(0+) = 0$, $v_1'(0+) = 1$, $v_2(0+) = \frac{1}{\omega_0} \not=0$ and 
$v_2'(0+) = -\frac{1}{\omega_0}\, \left(\frac{a(0)}{2}+\frac{\omega_1}{2\,\omega_0}\right)$ and therefore 
$$ 
        \omega(0+) = \omega_0 > 0  \qquad\mbox{and}\qquad    \omega'(0+) = \omega_1\,.
$$  
This concludes item a). \\
b) Now let $v_2^2(0) - v_1^2(0)>0$. We note that $\epsilon$ in~(\ref{ansatzomega2}) can be obtained 
from~(\ref{ansatzomega1}) if $v_1$ is replaced by $\i\,v_1$. Because $v_1$ satisfies equation 
$w'' + a\,w' + b\,w =0$, $\i\,v_1$ satisfies the same equation and thus identity~(\ref{Xomega}) is also true 
for this case. 
Similar as in Item a), the definition of $v_1$ and $v_2$ imply 
$$ 
       \i\,\epsilon(0+)   = \omega(0+) =  \i\,\omega_0 > 0  \qquad\mbox{and}\qquad   
      (\i\,\epsilon)'(0+) = \omega'(0+) = \i\,\omega_1\,,
$$
which concludes the proof. 
\end{proof}

\begin{theo}\label{theo:oscillmain}
Let $\varphi,\,\psi\in\R$, $a,\,b:[0,\T)\to \R$ be continuous and $\omega$ be as in Lemma~\ref{lemm:oscillmain} 
for given $\omega_0>0$ and $\omega_1\in\R$. Then the frequency function $\omega$ is well-defined on $[0,\T)$ 
with $\omega(0)=\omega_0$ and $\omega'(0)=\omega_1$, and the solution of~(\ref{2orderode}) is given 
by~(\ref{GenSolOscill}) with $\alpha$ and $\varrho$ defined by~(\ref{deftildeop}).  
\end{theo}

\begin{proof}
\emph{a) Existence}. Let $v_1$, $v_2$, $I_\omega$ and $I_\epsilon$ be defined as in Lemma~\ref{lemm:oscillmain}. 
The existence of the frequency function on $I_\omega$ or $I_\epsilon$, respectively, 
follows from Lemma~\ref{lemm:oscillmain}. It is clear that each set $I_\omega$ and $I_\epsilon$ differs from 
$[0,\T)$ by a countable set of real numbers. In item c) and d) below, we show with the help of the representation 
formula that $v_1^2 + v_2^2$ or $v_2^2-v_1^2$ has no zeros on $[0\,T)$ if $v_1^2(0) + v_2^2(0)>0$ or  
$v_2^2(0)-v_1^2(0)>0$ holds, respectively. \\ 
\emph{b) Solution formula}. 
Let $\lambda := \alpha \pm \i\,\omega$ with $\alpha:= \frac{a}{2} + \frac{\omega'}{2\,\omega}$. 
Because $\omega$ satisfies condition~(\ref{defomega}), it follows that $\lambda$ satisfies the 
(integrated) characteristic equation of~(\ref{2orderode}), i.e. 
$$
   -\lambda(t) + \lambda(0) + \int_0^t \lambda^2(s)\,\d s 
      - \int_0^t a(s)\,\lambda(s)\,\d s + \int_0^t b(s)\,\d s = 0\,.
$$ 
Thus
$$
   v_1(t) :=  \exp\left( -\tilde\alpha(t) - \i\,\tilde\omega(t)\right)  
                \quad\mbox{and}\quad
   v_2(t) :=  \exp\left( -\tilde\alpha(t) + \i\,\tilde\omega(t)\right)
$$
are solutions of the equation in~(\ref{2orderode}). Because these solutions are \emph{linear independent} for 
non-vanishing $\omega$, the solution of~(\ref{2orderode}) reads as follows 
$$
  v = \left[ C_1\,\exp( -\i\,\tilde\omega ) 
                + C_2\,\exp( \i\,\tilde\omega ) \right]\,\varrho 
  \qquad\mbox{with}\qquad \varrho := \exp( -\tilde\alpha) \,. 
$$
From this together with $v(0)=\varphi$ and $v'(0)=\psi$, 
we get
$$
    C_1 = \frac{\varphi}{2} + \i\,\frac{\psi+\alpha(0)\,\varphi}{2\,\omega(0)} 
             \qquad\mbox{and}\qquad
    C_2 = \frac{\varphi}{2} - \i\,\frac{\psi+\alpha(0)\,\varphi}{2\,\omega(0)} \,, 
$$
which yields the claimed solution formula. \\ 
\emph{c) $v_1^2+v_2^2$ has no zeros on $(0,\T)$ if $v_1^2(0) + v_2^2(0)>0$}. From the solution formula 
and the definition of $v_j$ for $j=1,\,2$, we get 
$$
   v_1(t) = \frac{\sin\left( \tilde\omega(t)\right)}{\omega(0)} \,\varrho(t) 
             \qquad\mbox{and}\qquad 
   v_2(t) = \frac{\cos\left( \tilde\omega(t)\right)}{\omega(0)} \,\varrho(t)   
$$
and thus $v_1^2(t)+v_2^2(t) \not= 0$ for all $t\in (0,\T)$. \\
\emph{d) $v_2^2-v_1^2$ has no zeros on $(0,\T)$ if $v_2^2(0) - v_1^2(0)>0$}. Similarly as before, the solution formula 
and the definition of $v_1$ and $v_2$ imply
$$
   v_1(t) = \frac{\sinh\left( \tilde\epsilon(t)\right)}{\omega_0} \,\varrho(t) 
             \qquad\mbox{and}\qquad 
   v_2(t) = \frac{\cos\left( \tilde\epsilon(t)\right)}{\omega_0} \,\varrho(t)   
$$
and hence $v_2^2(t)-v_1^2(t) \not= 0$ for all $t\in (0,\T)$. 
This concludes the proof.  
\end{proof}

We now present an example of a purely frequency dependent damped oscillation.  

\begin{exam}
Let $\T\in (0,\infty]$, $\omega_0>0$, $\omega_1\in\R$ and $w:(0,\infty)\to [0,\infty)$ be a differentiable 
function that is increasing and that satisfies $w(0+)=\omega_0$, $w'(0+)=\omega_1$ and 
$\lim_{t\to \T} w(t) = \infty$. Moreover, let 
$$
   a := \frac{w'}{w} 
               \qquad\mbox{and}\qquad
   b := w^2 + a' \,.
$$ 
Then $a$, $b$ and $\omega:=w$ satisfy condition~(\ref{defomega}) and 
$\frac{\omega_0}{\omega(t)} = \exp\left( - \int_0^t a(s)\,\d s \right) = \varrho(t)$ for $t\in (0,\T)$ 
and consequently the solution of~(\ref{2orderode}) is given by  
$$
   v(t) = \varphi\,\frac{\omega_0}{\omega(t)}\,\cos\left( \tilde\omega(t)\right) 
           + \left( \psi + \frac{\omega_1}{\omega_0}\,\varphi\right)\,
                 \frac{\sin\left( \tilde\omega(t)\right)}{\omega(t)}  \,.
$$ 
Because of $\lim_{t\to \T}\omega(\T-) = \infty$, it follows that $\lim_{t\to \T} v(t) = 0$. 
This example shows that the relaxation function $\varrho$ may depend only on the frequency. 
\end{exam}

Before we consider the second main theorem in this paper, we recall that the function 
$$
    v(t) := \int_0^t f_0(s)\,\exp\left( -\int_0^{t-s} \alpha(r+s) \,\d r \right)\,\d s \,H(t) 
      \quad\mbox{for}\quad t\in\R 
$$
for continuous $\alpha$ and $f_0$ with $\supp(f) \subset (0,\infty)$ is the unique solution of 
$$
   v' + \alpha\,v = f \quad\mbox{on}\quad \R \qquad\mbox{with}\qquad v|_{t<0} = 0\,,
$$ 
due to $\frac{\d }{\d t} \left( \int_0^{t-s} \alpha(r+s) \,\d r \right) = \alpha(t)$ and $f_0(0)=0$. 
This is a special case of the following result:

\begin{theo}\label{theo:oscillmainb}
Let $a$, $b$, $\omega$ be as in Lemma~\ref{lemm:oscillmain} and $\alpha$, $\varrho$ be as in 
Theorem~\ref{theo:oscillmain}. Then the solution of~(\ref{2orderode}) with $\varphi = \psi = 0$ and continuous 
$f$ satisfying $\supp(f) \subset (0,\infty)$ is given by 
\begin{equation*}
   v(t) = \int_0^t f_0(s)\,\frac{\sin\left( \tilde\omega_s(t-s)\right)}{\omega_s(0)} \,\varrho_s(t-s) \,\d s \,H(t)
     \qquad\mbox{for}\qquad t\in\R\,,
\end{equation*}
where $\omega_s := \omega(\cdot+s)$, $\alpha_s := \alpha(\cdot+s)$ and 
$\varrho_s := \exp( -\tilde\alpha_s(\cdot-s) )$. 
\end{theo}

\begin{proof}
It is clear that the solution of 
\begin{equation}\label{2orderode1}
\begin{aligned}
    v'' + a\,v' + b\,v = f_0 \quad\mbox{on}\quad \R \quad\mbox{with} \quad v|_{t<0} = 0 
    \quad\mbox{and} \quad f_0|_{t<0} = 0\,,
\end{aligned}
\end{equation}
is equal to $ v(t) := \int_0^t v_s\,\d s\, H(t)$ for $t\in\R$, where $v_s$ solves  
\begin{equation}\label{2orderode1}
\begin{aligned}
    v_s'' + a\,v_s' + b\,v_s = f_0(s)\,\delta(t-s) \quad\mbox{on}\quad \R \qquad\mbox{with} \quad v_s|_{t<s} = 0\,.
\end{aligned}
\end{equation}
Note that the right hand side $f_0(s)\,\delta(t-s)$ corresponds to the "initial data" 
$v_s(s) = 0$ and $v_s'(s) = f_0(s)$. The claimed formula follows from this and 
$v_s(t) =  f_0(s)\,\frac{\sin\left( \tilde\omega_s(t-s)\right)}{\omega_s(0)} \,\varrho_s(t-s)\,H(t-s)$. 
\end{proof}

\subsection{Classification and examples of oscillations}

The frequency of a \emph{standard} oscillation, i.e. the solution of~(\ref{2orderode}) 
with constant $a$ and $b$, is defined by $\omega = \sqrt{b - \frac{a^2}{4}}$. With this definition the 
\emph{weakly dissipative} oscillation the \emph{aperiodic limit} and \emph{creeping} can be classified by 
$\omega^2>0$, $\omega^2=0$ and $\omega^2<0$, respectively. This motivates:

\begin{defi}\label{defi:ClassOscill}
Let $\T$, $a$, $b$, $\omega$ be as in Definition~\ref{defi:frequency}, $I\subseteq [0,\T)$ and 
$v$ be the solution of~(\ref{2orderode}) on $(0,\T)$. 
\begin{itemize}
\item [1)] We call $v$ a \emph{weakly dissipative oscillation} on $I$, if 
           $\omega^2(t)>0$ for $t\in I$. 

\item [2)] We call the solution $v$ an \emph{aperiodic limit} on $I$ if $\omega^2(t)=0$ for $t\in I$. 

\item [3)] We call $v$ a \emph{creeping process} on $I$, if $\omega^2(t)<0$ for $t\in I$.

\item [4)] We call $v$ a \emph{mixed oscillation}, if none of the previous items are true on the whole 
           time interval $(0,\T)$. 

\end{itemize}
\end{defi}

The following example shows that a solution of $v''+a\,v'+b\,v=0$ with $a>0$ need not be damped.

\begin{exam}\label{exam:dissipativenotdamped}
Let $\T\in(0,\infty]$, $\omega:(0,\T)\to (0,\infty)$ be differentiable with $\omega>0$, 
$a:=-\frac{\omega'}{\omega}$ and $b:(0,\T)\to (0,\infty)$ be defined 
by~\footnote{If $\omega$ is twice differentiable, then $b = \omega^2 + \left(\frac{\omega'}{\omega}\right)'$.}
$
    \int_0^t \omega^2(s) \,\d s + \frac{\omega'}{\omega} (t) - \frac{\omega'}{\omega}(0)
          = \int_0^t b(s) \,\d s  \,.
$
Then $\alpha$ vanishes and the solution of problem~(\ref{2orderode}) reads as follows  
$v = \varphi\,\cos(\tilde\omega) + \psi\,\frac{\sin(\tilde\omega)}{\omega(0)}$. Here we assume that $\varphi$ and 
$\psi$ do not both vanish. 
Although $v$ solve the "damped" oscillation equation with non-vanishing 
$a$ and $b$, it does not describe a damped oscillation. However, its frequency function is not constant. \\
For example, if $\omega(t) = \omega_0\,e^{-\tilde \eta\,t}$ for $t\in [0,\T)$, where $T\in(0,\infty]$ and  
$\eta$, $\eta'$ are positive monotonic increasing functions with $\lim_{t\to\T-} \eta(t) = \infty$, 
then $a = \eta$ and $b = \omega^2 + \eta'$. We see that $a$ is monotonic increasing with 
$\lim_{t\to\T-} a(t) = \infty$ and that $b$ is monotonic increasing on $(\T_0,\T)$ for some $\T_0>0$ and  
$\lim_{t\to\T-} b(t) = \infty$. Moreover, the oscillation $v$ does not stop at $\T$, i.e. not both 
numbers $v(\T)$ and $v'(\T)$ vanish. 
\end{exam}

In the next two example, we model weakly dissipative oscillations with a finite stopping time $\T>0$. 
For appropriate initial frequency ($\omega_0=0$ or $\omega_0\in\i\,\R$), each of these examples reduces 
to an aperiodic limit  or creeping process. But these cases we will not be further investigated.

\begin{exam}\label{Exam:WDiss01}
Let $\T$, $a_0$, $b_0$ and $\omega_0$ be positive, $a_1\in (-\infty,a_0]$, 
$$
     a(t) := \frac{a_0\,\T}{\T-t} - \frac{a_1}{\T}\,t\,,\quad
     \omega(t) := \frac{\omega_0\,\T}{\T-t}  \quad\mbox{with}\quad 
        \omega_0 := \sqrt{b_0 - \left(\frac{a_0}{2}+\frac{1}{2\,\T}\right)^2}
$$
and
\begin{equation}\label{helpb01}
     b(t) := \frac{b_0\,\T^2}{(\T-t)^2} 
             + \frac{a_1^2}{4\,\T^2}\,t^2
             - \frac{a_1\,a_0}{2\,(\T-t)}\,t  
             - \frac{a_1}{2\,\T} 
\end{equation}
for $t\in [0,\T)$. Then $a$, $b$ and $\omega$ satisfy identity~(\ref{defomega}) and the solution 
of~(\ref{2orderode}) describes an oscillation that is weak dissipative or creeping if 
$\omega_0>0$ or $\omega_0 = \i\,\epsilon_0$ with $\epsilon_0>0$, respectively.  
For the special case $a_1=a_0$, we have $a'(0)=0$. Here we focus on the weak dissipative case. 
According to Theorem~\ref{theo:oscillmain}, the solution of~(\ref{2orderode}) reads as follows  
\begin{equation*}
\begin{aligned}
   v(t) &= \left[ \varphi\,\cos\left( \omega_0\,\T\,\log\left( 1- \frac{t}{\T} \right) \right) \right. \\
        &\qquad \left. 
                 - \left( \psi + \left( \frac{a_0}{2}+\frac{1}{2\,\T}\right)\,\varphi\right)\,
                 \frac{\sin\left( \omega_0\,\T\,\log\left( 1- \frac{t}{\T} \right) \right)}{\omega_0}          
          \right]\,\varrho(t)
\end{aligned}
\end{equation*}
with
$$
     \varrho(t) 
       = \sqrt{ \left( 1 - \frac{t}{\T} \right)^{1+a_0\,\T} }\,
         \exp\left( \frac{a_1}{4\,\T}\,t^2 \right)
$$ 
for $t\in[0,\T]$. Here we have used that $a(0)= a_1$, $\frac{\omega'(0)}{2\,\omega(0)} = \frac{1}{2\,\T}$,  
$\omega(0) = \omega_0$, 
$\int_0^t \frac{1}{\T-s}\,\d s = - \log\left( 1- \frac{t}{\T} \right)$ and 
$\alpha(t) = \frac{1+a_0\,\T}{2\,(\T-t)} - \frac{a_1}{2\,\T}\,t$. 
We note that 
\begin{itemize}
\item $\alpha'$ is strictly increasing if $a_1 < a_0 + \frac{1}{\T}$ (above we required $a_1<a_0$) and 
      thus $\varrho$ is strictly decreasing,

\item $\varrho$ is a linear (and decreasing) if $a_0=\frac{1}{\T}$ and $a_1=0$; 

\item if "$\T =\infty$", then $a = a_0$, $b = b_0$ and 
      $\omega_0 = \sqrt{b_0 - \frac{a_0^2}{4}}$, more precisely, for the limit $\T\to\infty$, 
      we obtain the standard oscillation with constant coefficients.

\end{itemize}

Two numerical examples are presented in Fig.~\ref{Fig:Exam:WDiss01} and Fig.~\ref{Fig:Exam:WDiss01B} for the setting 
(i) $\T=10$, $a_0\,\T=0.1$, $a_1=0$, $b_0\,\T=20$, $\varphi=\psi=1$ and (ii) $\T=10$, $a_0\,\T=5$, $a_1=0$, 
$b_0\,\T=1000$, $\varphi=\psi=1$, respectively.

\begin{figure}[!ht]
\begin{center}
\includegraphics[height=5cm,angle=0]{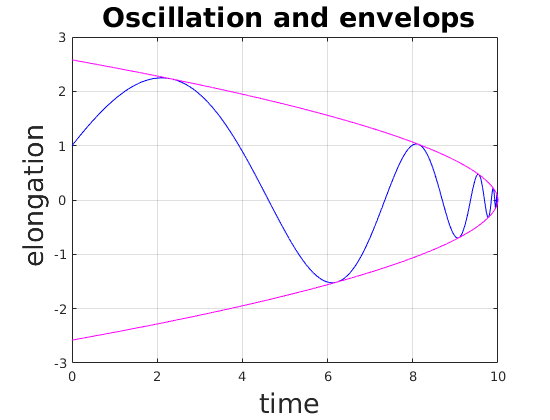}
\includegraphics[height=5cm,angle=0]{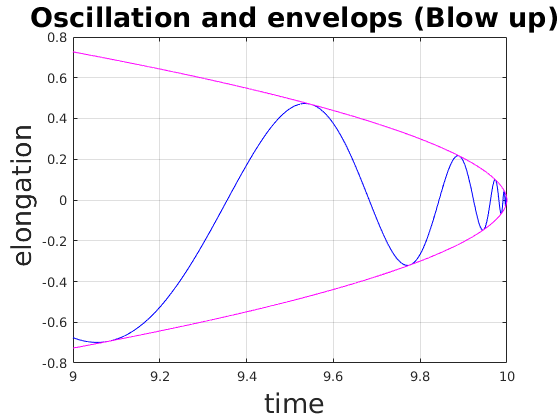}
\includegraphics[height=5cm,angle=0]{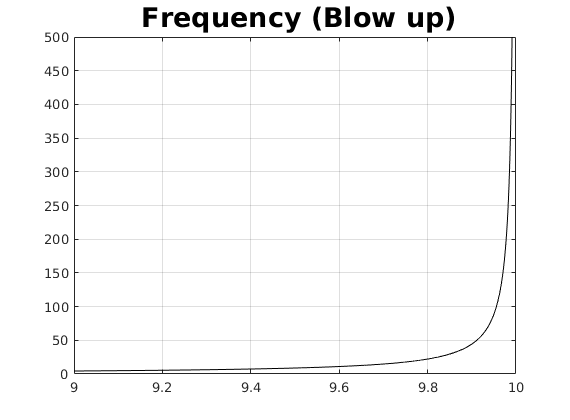}
\includegraphics[height=5cm,angle=0]{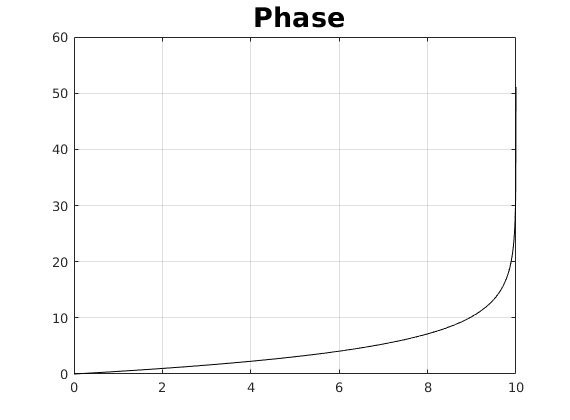}
\end{center}
\caption{Visualization of the oscillation satisfying $v'' + \frac{v'}{10\,(10-t)} + \frac{20\,v}{(10-t)^2}=0$ with 
$v(0) = v'(0) = 1$ (cf. Example~\ref{Exam:WDiss01}). Here the function 
$t\mapsto \tilde \omega(t) := \int_0^t \omega(s)\,d s$ is called \emph{phase function}. 
}
\label{Fig:Exam:WDiss01}
\end{figure}

\begin{figure}[!ht]
\begin{center}
\includegraphics[height=5cm,angle=0]{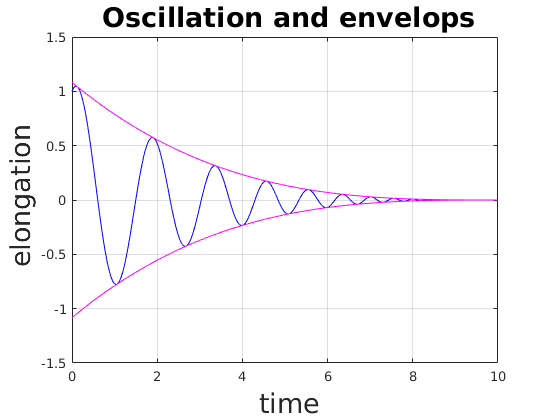}
\includegraphics[height=5cm,angle=0]{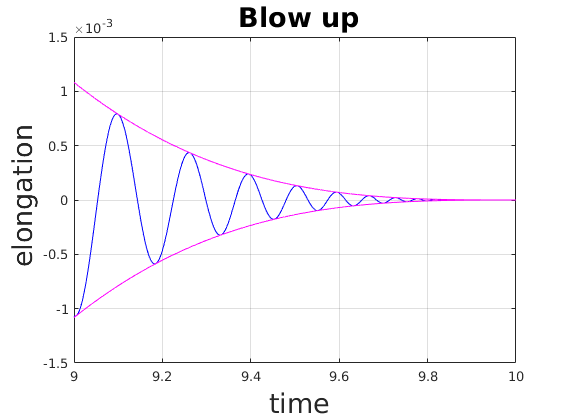}
\end{center}
\caption{Visualization of the oscillation satisfying $v'' + \frac{5\,v'}{10-t} + \frac{10^3\,v}{(10-t)^2}=0$ with 
$v(0) = v'(0) = 1$ (cf. Example~\ref{Exam:WDiss01}). 
}
\label{Fig:Exam:WDiss01B}
\end{figure}
\end{exam}

The previous example can be generalized as follows.

\begin{exam}
Let $\T$, $a_0$ and $b_0$ be positive, $n\in\N\backslash\{1\}$ and $a_1\in (-\infty,n\,a_0]$. Then 
identity~(\ref{defomega}) is satisfied for $a$, $b$ and $\omega$ defined by
$$
     a(t) := \frac{a_0\,\T^n}{(\T-t)^n} - \frac{a_1}{\T}\,t\,,\quad
     \omega(t) := \frac{\omega_0\,\T^n}{(\T-t)^n}  \quad\mbox{with}\quad 
        \omega_0 := \sqrt{ b_0 - \frac{a_0^2}{4} }
$$
and
$$
     b(t) := \frac{b_0\,\T^{2\,n}}{(\T-t)^{2\,n}} 
             + \frac{n\,a_0\,\T^n}{2\,(\T-t)^{n+1}}
             + \frac{n\,(2-n)}{4\,(\T-t)^2}
             - \frac{a_1\,a_0\,\T^{n-1}}{2\,(\T-t)^n}\,t  
             + \frac{a_1^2}{4\,\T^2}\,t^2
             - \frac{a_1}{2\,\T} 
$$
for $t\in [0,\T)$. Again, the solution $v$ of~(\ref{2orderode}) describes an oscillation that is weak 
dissipative or creeping if $\omega_0>0$ or $\omega_0 = \i\,\epsilon_0$ with $\epsilon_0>0$, respectively. 
We focus on the weak dissipative case and obtain 
\begin{equation*}
\begin{aligned}
   v(t) &= \left[ \varphi\,\cos\left( \frac{\omega_0\,\T^{2\,n-1}}{n-1}\, 
                                       \left[ \left( 1-\frac{t}{\T}\right)^{1-n} - 1 \right] \right) \right. \\
        &\quad \left. 
                 + \left( \psi + \left(\frac{a_0}{2}+\frac{n}{2\,\T}\right)\,\varphi\right)\,
                 \frac{\sin\left( \frac{\omega_0\,\T^{2\,n-1}}{n-1}\,
                           \left[ \left( 1-\frac{t}{\T}\right)^{1-n} - 1 \right] \right)}{\omega_0}          
          \right]\,\varrho(t)
\end{aligned}
\end{equation*}
with 
$$
  \varrho(t) = \left( 1-\frac{t}{\T} \right)^{\frac{n}{2}} \,
               \exp\left( \frac{a_1}{4\,\T}\,t^2 
                         - \frac{a_0\,\T^{2\,n-1}}{2\,(n-1)}\,\left( \left( 1-\frac{t}{\T} \right)^{1-n} -1 \right) \right)\,.
$$
If $a_1 \leq n\,a_0$ (as in the above assumption), then $\alpha'$ is strictly increasing and therefore 
$\varrho$ is strictly decreasing. For a numerical example see Fig.~\ref{Fig:Exam:WDiss01B}. 
\end{exam}

\begin{rema}
Apart from the fact that the representation formula enables us to solve any normalized second order ode 
(with nonconstant coefficients), it is also a very useful tool to model various dissipative oscillations, 
because it reveals the interplay between the coefficients $a$, $b$, and the damping constant $\alpha$ (or equivalently the relaxation function) and the frequency function $\omega$. 

We note that modeling of dissipative oscillations may also involve some requirement about its energy 
behaviour, but this is beyond the scope of this paper.
\end{rema}

\section{Modeling dissipative waves with dissipative oscillations}
\label{sec-DissWaves}

For pure frequency dependent dissipation, the  standard dissipative wave equation (cf.~\cite{KoSc12}) reads 
as follows
$$
    \left(D_{\alpha} + \frac{1}{c_0}\,\frac{\partial}{\partial t} \right)^2 u 
      - \Delta u 
                = f\, \quad\mbox{with}\quad u|_{t<0} = 0 \,,
$$ 
where $\alpha:\R\to\C$ is a given dissipation (or complex attenuation) law. Here the operator $D_{\alpha}$ 
is defined by $D_{\alpha} f = \F^{-1}\left( \alpha\,\F(f)\right)$, where $\F(f)$ denotes the 
Fourier transform of a tempered distribution $f$. We call a function $\alpha:\R\to\C$ a \emph{dissipation law} 
if $\Re(\alpha)$ is nonnegative and even and $\Im(\alpha)$ is odd. We do not assume that the attenuation 
law $\Re(\alpha)$ is monotonically increasing for positive frequencies and $\Re(\alpha)(0)=0$. 
We also denote the Fourier transform of $f$ by $\hat f$. 
More facts about wave dissipation can be found  in~\cite{NacSmiWaa90,Sz94,Sz95,KiFrCoSa00,WaHuBrMi00,We00,HanSer03,CheHolm04,WaMoMi05,PatGre06,KoSc12}. 
The fundamental solution of the above equation is given by~(\ref{modelG}) with~(\ref{modelKR}).

We call $\left( \K_R \right)_{R\geq 0}$ a \emph{semigroup of dissipative oscillations} if there exists a dissipation law $\alpha$ such that $\hat\K_R = e^{-\alpha\,R}$ for 
$R\geq 0$ and $\K_R(t) = 0$ for $t<0$ and $R>0$. Then we have (i) $\K_{R_1+R_2} = \K_{R_1} *_t \K_{R_2}$ for $R_1,R_2>0$ and (ii) $\lim_{R\to 0} \K_R = \delta$ in the distributive sense. 
As described in the introduction, a dissipative spherical wave $G$ is of the form~(\ref{modelG}) for given 
$\left( \K_R \right)_{R\geq 0}$.

Moreover, we call $\K_1$ a \emph{relaxation function} if $\K_1$ is a positive, monotonic 
decreasing causal function (, i.e. $\K_1|_{t<0}=0$) that is an element of $L^1(\R)$.

\begin{prop}\label{prop:alpha1}
Let $\K_1$ be a relaxation function satisfying $\|\K_1\|_{L^1(\R)} \leq \frac{1}{\sqrt{2}}$. 
Then there exists a dissipation law $\alpha$ such that $\hat \K_1 := e^{-\alpha}$.  
\end{prop}

\begin{proof} 
Because $\K_1\in L^1(\R)$, its Fourier transform exists and 
$$
    \Re(\hat\K_1)(\omega) 
       = \int_0^\infty \K_1(t)\,\cos(2\,\pi\,\omega\,t) \,\d t  
         \quad\mbox{and}\quad  
    \Im(\hat\K_1)(\omega) 
       = \int_0^\infty \K_1(t)\,\sin(2\,\pi\,\omega\,t) \,\d t \,.        
$$ 
We have to show that there exists a dissipation law $\alpha$ satisfying $\hat\K_1 = e^{-\alpha}$. 
Such a dissipation law $\alpha$ exists, if $\Re(\alpha) \equiv -\log |\hat\K_1| \in [0,\infty)$, 
$\Re(\alpha)$ is even and $\Im(\alpha) \equiv \frac{\hat\K_1}{|\hat\K_1|}$ is odd. 
That $\Re(\alpha)$ is even and $\Im(\alpha)$ is odd follows at once from the fact that $\K$ is real 
valued. Thus it remains to show that $|\hat\K_1|\in (0,1]$ is true. \\
Visualizing the graph of $t \mapsto \K_1(t)\,\sin(2\,\pi\,\omega\,t)$ shows that $\Im(\hat\K_1)(\omega) > 0$ 
for $\omega\in\R$, because $\K_1$ is positive and decreasing on $(0,\infty)$. As a consequence, 
$|\hat\K_1| = \sqrt{ \Re(\hat\K_1)^2 + \Im(\hat\K_1)^2 }$ does not vanish on $\R$. 
It remains to show that $\Re(\hat\K_1)^2 + \Im(\hat\K_1)^2$ is bounded by $1$. But this property is true, 
due to  
$$
    \Re(\hat\K_1)^2 + \Im(\hat\K_1)^2 
      \leq 2\,\int_0^\infty \K_1(t)^2\,\d t
      =  2\,\|\K_1\|_{L^1(\R)}^2
$$ 
and the assumption  
$\|\K_1\|_{L^1(\R)} \leq \frac{1}{\sqrt{2}}$. \\
Well-definedness of $\K_R$ for $R\geq 0$. Because the complex exponential function satisfies 
$(e^{z})^R=e^{z\,R}$ for $z\in\C$ and $R\in\R$, it follows that $\K_R = \F^{-1}(e^{-\alpha\,R}) = \K_1^R$ exists 
and is well-defined for $R\geq 0$. 
\end{proof}

Two of the simplest oscillations are of the form $\K_1 := \varrho$ and 
$\K_1 := \varrho\,\cos(2\,\pi\,\omega_0\,t)$, where $\varrho$ is a relaxation and $\omega_0$ is a constant 
frequency. The first one was already considered in the previous proposition, now we consider the second 
one.
For convenience, we define $[0,\T]$ to be $[0,\T)$ if $\T=\infty$.

\begin{prop}\label{prop:Semigroup2}
Let $\omega_0$ be a positive constant, $\varrho\in L^1(\R)$ satisfy 
$\|\varrho\|_{L^1(\R)} \leq \frac{1}{\sqrt{2}}$ and $\supp(\varrho) = [0,\T]$ for some 
$\T\in (0,\infty]$. Then $\K_1 := \varrho\,\cos(2\,\pi\,\omega_0\,t)$ induces a \emph{semigroup of 
dissipative oscillations} by $\hat \K_R := \hat\K_1^R$ satisfying $\supp(\K_R) = [0,R\,\T]$ for 
$R\geq 0$.
\end{prop}

\begin{proof} 
Because $\varrho\in L^1(\R)$, $\K_1\in L^1(\R)$ and its Fourier transform exists. 
According to the proof of Proposition~\ref{prop:alpha1}, it remains to show that 
$$
        \Re(\hat\K_1)^2(\omega) + \Im(\hat\K_1)^2(\omega) \in (0,1]
       \qquad\mbox{for}\qquad   \omega\in\R\,. 
$$ 
From $|\Re(\hat\K_1)|,\, |\Re(\hat\K_1)| \leq \int_0^\infty |\varrho(t)|\,\d t $, we infer 
$$
    \Re(\hat\K_1)^2 + \Im(\hat\K_1)^2 
      \leq 2\,\|\varrho\|_{L^1(\R)}^2 \leq 1
$$ 
and consequently $|\hat\K_1(\omega)| \leq 1$. Now we show that $|\hat \K_1|$ does not vanish on $\R$.  
From 
$$
   \cos(2\,\pi\,\omega_0\,t) 
     = \frac{1}{2}\,\left( e^{\i\,2\,\pi\,\omega_0\,t} + e^{-\i\,2\,\pi\,\omega_0\,t} \right) 
     =: \frac{1}{2}\,(h_1(t) +h_2(t))
$$
and 
$$
   \delta(\omega \mp \omega_0) = \F(e^{\pm\i\,2\,\pi\,\omega_0\,t})(\omega) \,,
$$
we infer 
$$
   \hat\K_1(\omega)
     = \frac{1}{2}\,\int_0^\infty \hat\varrho(\omega-s)\, (\hat h_1 + \hat h_2)(s) \,\d s 
     = \frac{1}{2}\,\left[\hat\varrho(\omega-\omega_0) 
                                   + \hat\varrho(\omega+\omega_0)\right]\,.
$$
According to the proof of  Proposition~\ref{prop:alpha1}, $\Re(\hat \varrho)$ is positive and thus 
$|\hat \K_1|$ does not vanish, too. \\ 
The last claim $\supp(\K_R) = [0,R\,\T]$ for $R\geq 0$ follows at once from Proposition~\ref{prop:alpha2} in 
the Appendix. Note, $\alpha$ is entire, due to $\supp(\K_1) = [0,\T]$. This concludes the proof. 
\end{proof}

The essential property that we required for the proof of the previous two Propositions was 
\begin{equation}\label{EssProp}
              |\hat\K_1|(\omega) > 0  \qquad\mbox{for}\qquad   \omega\in\R \,.
\end{equation} 
Let $\T\in (0,\infty)$, $C_\omega(t):= \cos(2\,\pi\,\omega\,t)$, $S_\omega(t):= \sin(2\,\pi\,\omega\,t)$ 
for $t\in[0,\T]$ and $\langle f,g \rangle := \int_0^\T f(t)\,g(t)\,\d t$ for $f,\,g\in L^1(0,\T)$. 
Then, geometrically, property~(\ref{EssProp}) means that 
\begin{itemize} 
\item [(EP)] $\K_1$ is not orthogonal to $C_\omega$ and $S_\omega$ for all $\omega\in\R$\,,
\end{itemize}
with respect to the scalar product $\langle \cdot,\cdot \rangle$. If $\T$ is infinite, then we may consider 
the scalar product on $C(0,\infty)$. In summary, we have

\begin{theo}
Let $\varrho$ be a relaxation function $\supp(\varrho) = [0,\T]$ for some $\T\in (0,\infty]$ satisfying 
$\|\varrho\|_{L^1(\R)} \leq \frac{1}{\sqrt{2}}$ and $\omega_*:\R\to\R$ be a real-valued frequency function. 
If property~(\ref{EssProp}) or equivalently property (EP) holds, then 
$\K_1$ induces a \emph{semigroup of dissipative oscillations} 
by $\hat \K_R := \hat\K_1^R$ satisfying $\supp(\K_R) = [0,R\,\T]$ for $R\geq 0$.
\end{theo}

In the two examples below, we shortly discuss two families of dissipative semigroups.

\begin{exam}\label{exam:Instr}
Let $a_0:=\frac{1}{\t}$ with $\t>0$ and $\K_R(t) = a_0\,\chi_+^{R-1}(a_0\,t) \,e^{-a_0\,t}$ 
for $R>0,\,t\in\R$, where (cf. Section~3.2 in~\cite{Ho03})
$$
  \;\chi_+^r(t) := \left\{ \begin{array}{ll}
                  \frac{t^r}{\Gamma(r+1)} \;  &  \; \mbox{for} \quad  t>0\\
                      0                   \;  &  \; \mbox{for} \quad  t\leq 0 
                 \end{array} \right. 
     \qquad \mbox{and}\qquad   r \in (-1,\infty)\,.
$$
Then $\K_R\in L^1(\R)$ for $R>0$ (cf. Exercise 6.3.18 (iii) in~\cite{St99}) and $\K_R\in L^2(\R)$ 
for $R>\frac{1}{2}$.  Moreover, let $G$ be defined as in~(\ref{modelG}). 
In Fig.~\ref{fig:ExInstr01}, we have visualized the qualitative behaviour of the "oscillations" $\K_R$ 
for various values of $R$. We see 
\begin{itemize}
\item [a)] if $R\in (0,1)$ then $\K_R$ has a pole at $t=0$ and is convex\footnote{This is not a proper  oscillation.},  

\item [b)] if $R=1$ then $\K_R(0+)$ exists and $\K_R$ is positive and convex, and   

\item [c)] if $R>1$ then $\K_R(0+)=0$, $\K_R$ is non-negative and has (exactly) one maximum at $t = (R-1)\,\t$. 
\end{itemize}
Thus there are three different "types" of oscillations depending on the distance $R$ from the origin of the wave 
which obey the dissipation law $\alpha/R$. 
\begin{figure}[!ht]
\begin{center}
\includegraphics[height=5.0cm,angle=0]{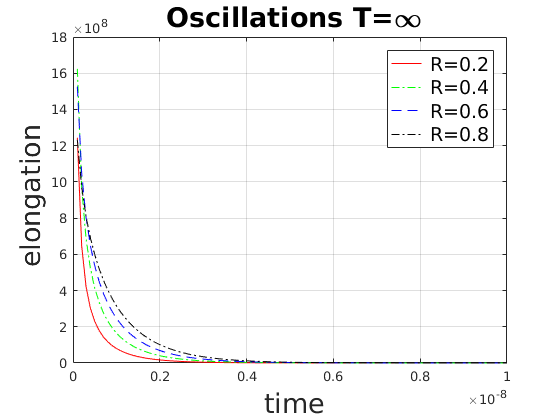}
\includegraphics[height=5.0cm,angle=0]{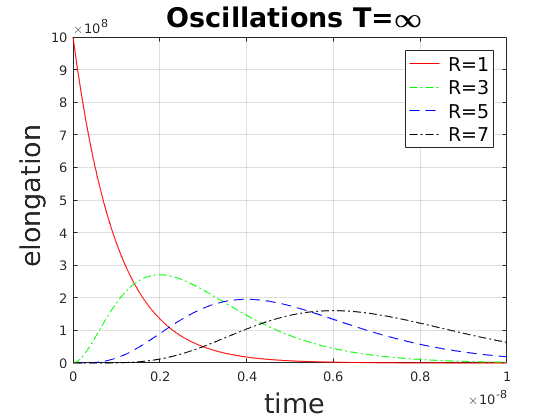}
\includegraphics[height=5.0cm,angle=0]{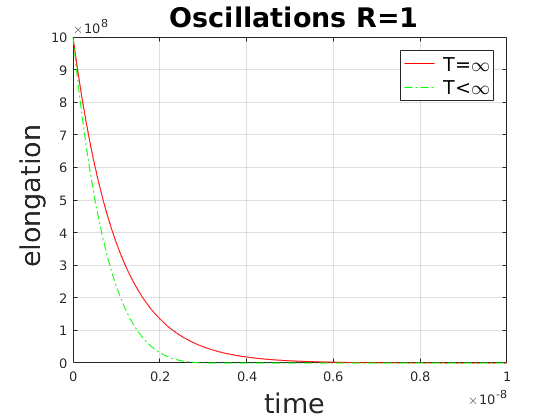}
\includegraphics[height=5.0cm,angle=0]{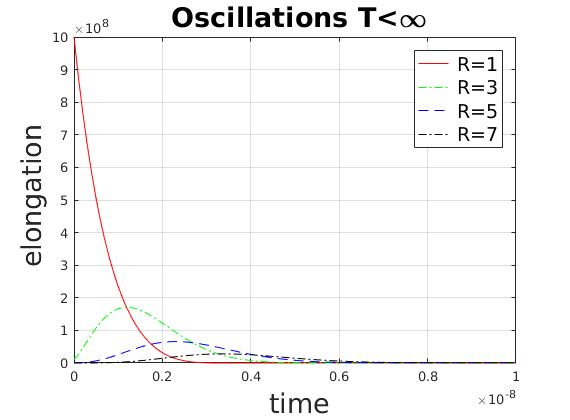}
\end{center}
\caption{Visualisation of the "oscillations" $\K_R$ from Example~\ref{exam:Instr} ($\T=\infty$) and 
Example~\ref{exam:L} ($\T<\infty$) with $m=2$ for $R\in\{ 0.2,\,0.4,\,0,6,\,0.8 \}$ and 
$R\in\{ 1,\,3,\,5,\,7 \}$, respectively. As parameters, we used $\t=10^{-9}\,\,\mbox{s}$ and 
$c_\infty=1481\,\frac{ \mbox{m} }{ \mbox{s} }$. }
\label{fig:ExInstr01}
\end{figure}
\end{exam}

Let us now present an example of a relaxation function with support in $[0,\T]$ that can be used to model oscillations with finite stropping time $\T$. 

\begin{exam}\label{exam:L}
For example, let $\rho$ be the function $\K_1$ from Example~\ref{exam:Instr} and $\K_1$ be defined by 
\begin{equation*}
     \K_1 (t) 
          := \rho(t)\,\frac{ 1 - \sum_{m=0}^k \frac{ \rho^{(m)}(\T) }{ \rho(t)  }\,\frac{(t-\T)^m}{m!} }
                           { 1 - \sum_{m=0}^k \frac{ \rho^{(m)}(\T) }{ \rho(0+) }\,\frac{( -\T)^m}{m!} }\, 
                  \chi_{[0,\T]}(t) \,,
\end{equation*}
which is $k$-times continuously differentiable on $(0,\infty)$ and has support in $[0,T]$. 
\end{exam}

\begin{exam}\label{exam:Instr2}
Let us shortly consider the semigroup defined by
$$
 \K_R(t) := \frac{t^{R-1}}{\t_0^R\,\Gamma(R)}\,\exp\left( -\frac{t}{\t_0} \right)\,\cos(\omega_0\,t) 
   \qquad\mbox{for}\qquad R\geq 0\,,
$$
where $\t_0$ and $\omega_0$ are positive constants. For $\omega_0=0$ the oscillation $\K_R$ reduces 
to the oscillation $\K_R$ from the previous example. In contrast to the previous examples, the set of 
extrem values of $\K_R$ is countable. Similarly as in Example~\ref{exam:Instr}, $\K_R$ has a pole at 
$t=0$ for $R<1$ and $\K_R$ has no pole for $1 < R$. 
\end{exam}

\section{Conclusion}

In this paper we performed an analysis of (normalized) second order odes with nonconstant coefficients 
and derive a solution formular for this type of problem. With this formula we get a much better 
understanding of dissipative oscillation and their respective models. The essential part of this formular 
is the existence of a unique (nonnegative) frequency function that satisfies an integro differential 
equation. If this integro differential equation is fully understood, then dissipative oscillation are also 
fully understand. Therefore we think that this equation should be analysed in detail by 
scientists working in the field of integro differential equation or similar fields. 
Moreover, we demonstrated that the formula is very useful to model or design dissipative oscillations and 
dissipative waves. As a novelty, we modeled dissipative waves that 
cause oscillations in each space point that stop after a finite time period.

\section{Appendix: Two applications of the Paley-Wiener-Schwartz Theorem}

For convenience, we summarize some notation about the Fourier transform and present a small application 
of the Paley-Wiener-Schwartz Theorem that is used in our analysis of dissipative waves with 
frequency dependent attenuation laws (cf. ~\cite{GaWi99,Ho03,DaLi92_5}).

We use the following form of the Fourier transform of an $L^1-$function $f$ 
$$
     \F(f)(\omega) := \hat f(\omega) := \int_\R f(t)\, e^{\i\,2\,\pi\,\omega\,t}\,\d t
     \qquad\mbox{for}\qquad \omega\in\R\,.
$$
Then the Convolution Theorem for $L^1-$functions reads as follows
$$
      \F(f *_t g) = \F(f)\,\F(g)
$$
and, if $f\in L^1(\R)$ is differentiable, then 
$$
         \F(f')(\omega) = (-\i\,2\,\pi\,\omega)\,\F(f)(\omega) \qquad\mbox{for}\qquad \omega\in\R\,.
$$
The inverse Fourier transform is denoted by $\check f$ and $\F^{-1}(f)$. 
In order to prove that a tempered distribution has support in $[0,\T]$ for some $\T\in(0,\infty]$, 
the following Theorem is essential.

\begin{theo}[Paley-Wiener-Schwartz] \label{theo:PWS}
$f\in \mathcal{S}'(\R)$ has support in $[0,\T]$ if and only if
\begin{itemize}
\item [(C1)] $z\in\C \mapsto \hat f(z)$ is entire and 

\item [(C2)] there exist constants $C$ and $N$ such that
$$
   |f(z)| \leq C\,(1+|z|)^N\,\exp\{ \sup_{t\in [0,\T]} (t\,\Im(z)) \}  \qquad \mbox{for}\qquad z\in\C\,.
$$
\end{itemize}
\end{theo}

Let $\alpha$ be a dissipation law (of a dissipative wave) and $\K_R$ be defined as in~(\ref{modelKR}). 
If $\supp(\K_1) = [0,\T]$ for some finite $\T>0$, then it is clear that $\supp (\K_R *_t \K_R) = [0,2\,T]$ 
holds. Now we show that the structure of the semigroup $(\K_R)_{R\geq 0}$ implies 
\begin{equation}\label{propKR}
    \supp (\K_R) = [0,R\,T]  \qquad\mbox{for}\qquad R\geq 0 \,.
\end{equation}

\begin{prop}\label{prop:alpha2}
Let $\K_1$ and $\alpha$ be as in Proposition~\ref{prop:alpha1} with the additional assumption 
$\supp\left( \K_1\right) = [0,\T]$ for some $\T\in (0,\infty)$. Then $\K_R$ defined as in~(\ref{modelKR}) 
satisfies~(\ref{propKR}).  
\end{prop}

\begin{proof} 
Let $\beta_1$ and $\beta_2$ be such that $\alpha(z) = \beta_1(z) + \i\,\beta_2(z)$ with $\beta_1(z),\beta_2(z)\in\R$ 
for each $z\in\C$. Then we have $\left| e^{-\alpha(z)\,r} \right|  =  e^{ -\beta_1(z)\,R }$ with $R:=|\x|$. 
According to the Paley-Wiener-Schwartz Theorem (Theorem~\ref{theo:PWS} above), we have 
$\supp(\K_1) = [0,\T]$  if and only if 
(i) $z\in\C \mapsto \hat \K_1(z)$ is entire and 
(ii) there exist constants $C$ and $N$ such that
$$
   e^{-\beta_1(z)} 
    \leq C\,(1+|z|)^N\,\exp\{ \sup_{t\in [0,\T]} (t\,\Im(z)) \}  \qquad \mbox{for}\qquad z\in\C\,.
$$
Because of $\supp\left( \K_1\right) = [0,\T]$ property (i) holds. The previous estimation for 
$e^{-\beta_1(z)}$ together with
\begin{equation}\label{helpPWtheo}
     e^{ -\beta_1(z)\,R } = \left( e^{ -\beta_1(z) } \right)^R 
          \qquad\mbox{and}\qquad
     R\,\sup_{t\in [0,\T]} (t\,\Im(z)) = \sup_{t\in [0,R\,\T]} (t\,\Im(z))
\end{equation}
for $R\geq 0$ implies 
$$
   e^{-\beta_1(z)\,R} 
    \leq C\,(1+|z|)^{R\,N}\,\exp\{ \sup_{t\in [0,R\,\T]} (t\,\Im(z)) \}  \qquad \mbox{for}\qquad z\in\C\,,
$$
i.e. $\supp(\K(R,\cdot)) = [0,R\,\T]$, which proves the claim. 
\end{proof}

\end{document}